\documentclass{article}

\usepackage{latexsym}
\usepackage{amssymb}
\usepackage{amsthm}
\usepackage{type1cm}        
\usepackage{makeidx}         
\usepackage{graphicx}        
\usepackage{multicol}        
\usepackage[bottom]{footmisc}

\usepackage{enumitem}

\usepackage{newtxtext}       %
\usepackage{newtxmath}       

\makeindex             

\newenvironment{@abssec}[1]{%
    \if@twocolumn

      \section*{#1}%
    \else

      \vspace{.05in}\footnotesize
      \parindent .2in
 {\upshape\bfseries #1. }\ignorespaces
    \fi}

    {\if@twocolumn\else\par\vspace{.1in}\fi}

\newenvironment{keywords}{\begin{@abssec}{\keywordsname}}{\end{@abssec}}

\newenvironment{AMS}{\begin{@abssec}{\AMSname}}{\end{@abssec}}

\newcommand\keywordsname{Key words}
\newcommand\AMSname{AMS subject classifications}
\newcommand\AMname{AMS subject classification}

\newcommand\setbld[2]{\left\{ #1 \;\middle |\; #2\right\}}

\newcommand\restr[2]{{
\left.\kern-\nulldelimiterspace 
#1 
\vphantom{|} 
\right|_{#2} 
}}
\newcommand\jump[1]{\left\llbracket #1\right\rrbracket}

\newtheorem{mainthm}{Theorem}

\newtheorem{theorem}{Theorem}[section]

\newtheorem{remark}[theorem]{Remark}
\newtheorem{definition}[theorem]{Definition}

\numberwithin{equation}{section}

\begin{document}

\title{A miscellanea of qualitative and symmetry properties of the solutions to the two-phase Serrin's problem}
\author{Lorenzo Cavallina \thanks{The author is partially supported by JSPS KAKENHI Grant Number JP22K13935 and  JP21KK0044, JP23H04459.}}
\date{}
%
%
\maketitle

\abstract{
This paper investigates the solutions to the two-phase Serrin’s problem, an overdetermined boundary value problem motivated by shape optimization. Specifically, we study the torsional rigidity of composite beams, where two distinct materials interact, and examine the properties of the optimal configurations (critical shapes) under volume constraints. We first show that such a shape optimization problem admits no local minimizers. Then, using the method of moving planes, we show that the solutions exhibit no extended or narrow branches (“tentacles”) away from the core. We then show that the outer boundary of a solution cannot exhibit flat parts and that the only configuration whose outer boundary contains a portion of a sphere is the one given by concentric balls. Finally, we establish that concentric balls are the only admissible configurations that solve the two-phase Serrin's problem for two distinct sets of conductivity values.}



\section{Introduction}

Let $\Omega$ be a bounded domain of $\mathbb{R}^N$ ($N\ge 2$) of class (at least) $C^{1,\gamma}$ for $\gamma\in(0,1]$ and let $D\subset\overline{D}\subset \Omega$ be a nonempty bounded open set. Throughout this paper, we will be concerned with the following overdetermined problem (which is a generalization of the famous Serrin's overdetermined problem \cite{serrin1971}):
\begin{equation}\label{odp}
    \begin{cases}
        -{\rm div}(\sigma\nabla u)=1\quad\text{in }\Omega,\\
        u=0\quad\text{on }\partial\Omega,\\
        \partial_n u = c \quad\text{on }\partial\Omega.
    \end{cases}
\end{equation}
Here $\partial_n$ denotes the normal derivative with respect to the outward unit normal vector $n\in C^{\gamma}(\partial\Omega,\mathbb{R}^N)$, $c$ is a given real constant, and $\sigma$ is the piece-wise constant function defined as 
\begin{equation}\label{sigma def}
    \sigma(x):=\begin{cases}
        \sigma_c \quad \text{if }x\in D\;\text{ (core)},\\
        1\quad \text{if }x\in \Omega\setminus D \;\text{ (shell)},
    \end{cases}
\end{equation}
for some given positive constant $\sigma_c\ne 1$. 
Notice that, problem \eqref{odp} is \emph{overdetermined} as we are imposing two distinct boundary conditions on $\partial\Omega$ at the same time. As a result, problem \eqref{odp} will only admit a solution $u$ for some \emph{suitable pairs} $(D,\Omega)$. Throughout this paper, when no confusion arises, we will also refer to such pairs $(D,\Omega)$ as the \emph{solutions} to the overdetermined problem \eqref{odp}. 

When $D$ is just an open set, one can give meaning to problem \eqref{odp} via the following weak formulation: we say that a function $u\in H_0^1$ solves \eqref{odp} if it satisfies 
 \begin{equation}\label{u weak form}
     \int_\Omega \sigma \nabla u\cdot\nabla\varphi = \int_\Omega \varphi + c\int_{\partial \Omega}\varphi
 \end{equation}
for all $\varphi\in H^1(\Omega)$.
Even though less immediate to see at first glance, the above formulation still yields an overdetermined problem, as we have also imposed the value of the trace of $u$ on the boundary $\partial\Omega$ by requiring $u\in H_0^1(\Omega)$. 
\begin{remark}[Reformulation as a transmission problem]
Independently of the regularity of the interface $\partial D$, equation $-{\rm{div}}(\sigma\nabla u)=1$ holds in the classical sense point-wise away from the interface. That is to say, $u$ satisfies $-\sigma_c\Delta u=1$ in $D$ and $-\Delta u= 1$ in $\Omega\setminus\overline{D}$ (this follows by localizing the equation in each phase by choosing suitable test functions in the weak form \eqref{u weak form}). Moreover, if $\partial D$ is sufficiently smooth, $u$ satisfies the following \emph{transmission conditions} at the interface:
\begin{equation*}
\jump{u}=0 \quad \text{on }\partial D, \quad 
\jump{\sigma \partial_n u}=0 \quad \text{on }\partial D,
\end{equation*}
where $\jump{\cdot}$ denotes the difference between the traces on $\partial D$ from the outside and the inside respectively, and $\partial_n$ denotes the normal derivative on $\partial D$.

\end{remark}

\begin{remark}[On the regularity of $\partial\Omega$]\label{rmk boundary}
Throughout this paper, we assume that $\partial\Omega$ is of class $C^{1,\gamma}$ for some $\gamma\in (0,1]$. Thus, by the classical Schauder boundary estimates (see \cite{kenig1994harmonic}), the solution $u$ of the Dirichlet problem given by the first two lines of \eqref{odp} is of class $C^{1,\gamma}$ in a neighborhood of the boundary $\partial\Omega$. In particular, if such $u$ also satisfies the Neumann boundary condition in \eqref{odp} (that is, if $(D,\Omega)$ solves the overdetermined problem \eqref{odp}), then it does so in the classical (point-wise) sense. As a result, the local regularity result \cite[Theorem 2]{kinderlehrer1977regularity} implies that $\partial\Omega$ is an analytic surface.
\end{remark}

Overdetermined problem \eqref{odp} has many physical interpretations. 
For instance, when Serrin originally introduced his overdetermined problem in \cite{serrin1971}, corresponding to the case $\sigma\equiv 1$ in \eqref{odp}, he gave the following fluid dynamical interpretation:

``\emph{the tangential stress on the pipe wall is the same at all points of the wall if and only if the pipe has a circular cross section.}"

\noindent Another interpretation  given in \cite{serrin1971} 
pertains to the linear theory of elasticity (see \cite{sokolonikoffElasticity}):

``\emph{when a solid straight bar is subject to torsion, the magnitude of the resulting traction which occurs at the surface of the bar is independent
of position if and only if the bar has a circular cross section.}"

\noindent The latter interpretation readily generalizes to the case described in \eqref{odp}, where the bar is a ``composite", that is, made by two materials rigidly bonded together according to the distribution given by \eqref{sigma def}. 

We remark that there is a natural reason for requiring the traction to be constant on the outer surface of the bar. 
Indeed, as it will be explained in further detail in section \ref{sec pf I}, such an overdetermined condition arises as a necessary condition for the shape optimization problem of maximizing the torsional rigidity functional $T_D(\cdot)$ under a volume constraint. Here 
\begin{equation*}\label{torsion}
  T_D(\Omega):= \int_\Omega \sigma |\nabla u |^2= \int_\Omega u ,  
\end{equation*}
where $u$, the so-called ``\emph{Prandtl stress function}", is defined as the solution to the boundary value problem
\begin{equation*}
    -{\rm div}(\sigma\nabla u)=1\quad\text{in }\Omega,\qquad 
        u=0\quad\text{on }\partial\Omega,
\end{equation*}
that is, $u$ the unique function in $H_0^1(\Omega)$ satisfying \eqref{u weak form} for all $\varphi\in H_0^1(\Omega)$. We remark that such a simplified description of torsional rigidity via Prandtl stress function is physically accurate only if the bar is sufficiently long (compared to the dimension of its cross-section) and we refer the interested reader to \cite{dondl2024phase}, where the authors fruitfully use this approach to study the multi-material composites found in the morphology of plant
stems. 

Finally, we briefly mention some works where similar overdetermined problems have been considered in the context of two-phase thermal conductors \cite{sakaguchi2016rendiconti,sakaguchibessatsu,CavallinaMagnaniniSakaguchi2021constantflowproperty}. 
\\

In what follows, we will introduce the main results of this paper.
As shown in \cite{cavallina2019stability}, the radial configurations $(D_0,\Omega_0)$ given by concentric balls are \emph{local maximizers} for $T_D$ under a volume constraint if $\sigma_c>1$ and are \emph{saddle shapes} if $\sigma_c<1$. The following theorem concerns the non-existence of local minimizers.     
\begin{mainthm}\label{thm I}
Let $(D,\Omega)$ be a solution to the overdetermined problem \eqref{odp}. Then, $(D,\Omega)$ is not a local minimizer in the sense of Definition \ref{def critical shapes}.    
\end{mainthm}

The remaining theorems concern some qualitative properties of $\partial\Omega$. In particular, the following theorem formalizes the intuitive notion that ``$\Omega$ follows the shape of $D$ without meandering too much". 

\begin{mainthm}\label{thm II}
 Let $(D,\Omega)$ be a solution to the overdetermined problem \eqref{odp}. Then $\Omega$ ``\emph{has no tentacles}", in the sense of Definition \ref{def tentacles}.  
\end{mainthm}

The following theorem shows that $\partial \Omega$ generally exhibits a non-trivial shape.

\begin{mainthm}\label{thm III}
 Let $(D,\Omega)$ be a solution to the overdetermined problem \eqref{odp}. Then $\partial\Omega$ does not contain any flat part. 
 Moreover, if $\Omega\setminus\overline{D}$ is connected and $D$ is an open set of class $C^2$ with finitely many connected components, then $\partial\Omega$ contains a portion of a sphere if and only if $(D,\Omega)$ are concentric balls.   
\end{mainthm}

Finally, the following theorem shows that the shape of the free boundary $\partial\Omega$ depends on the choice of the constant $\sigma_c$ unless $(D,\Omega)$ are concentric balls. 

\begin{mainthm}\label{thm IV}
 Let $\Omega\setminus\overline{D}$ be connected and assume that the boundary of the open set $D$ has zero Lebesgue measure and satisfies $\partial D = \partial \overline{D}$. Under these assumptions, $(D,\Omega)$ is simultaneously a solution to the overdetermined problem \eqref{odp} for two different values of $\sigma_c$ if and only if $(D,\Omega)$ are concentric balls. 
\end{mainthm}

This paper is organized as follows.
In section \ref{sec pf I}, we give the precise definition of \emph{local maximizers}, \emph{local minimizers}, and \emph{saddle shapes} and show Theorem \ref{thm I} by making use of the asymptotic behavior of the eigenvalues of a two-phase Neumann-to-Dirichlet operator. In section \ref{sec pf II}, we give the precise definition of ``\emph{having a tentacle}" and show Theorem \ref{thm II} by adapting the classical proof of Serrin to the two-phase setting. In section \ref{sec pf III}, we show Theorem \ref{thm III} by employing the analyticity of the boundary $\partial\Omega$ and a symmetry result for two-phase overdetermined elliptic problems by Sakaguchi \cite{sakaguchibessatsu}. Finally, section \ref{sec pf IV} is devoted to the proof of Theorem \ref{thm IV}, where the desired symmetry is obtained by applying Serrin's symmetry result \cite{serrin1971} to a suitably defined auxiliary function. 

\section{Proof of Theorem \ref{thm I}}\label{sec pf I}

Let $(D,\Omega)$ be a solution to \eqref{odp}. As mentioned in the introduction, $\Omega$ is a critical shape for the torsional rigidity functional $T_D$ under a volume constraint. In what follows, we will give the precise definition of the three types of critical shapes: ``\emph{local maximizer}", ``\emph{local minimizer}", and ``\emph{saddle shape}".

For a given $\Omega$, let $\mathcal{A}$ denote the space of admissible perturbations of $\Omega$ that fix the volume of $\Omega$:
\begin{equation}\label{cA}
    \mathcal{A}:=\setbld{\Phi\in C^2\left([0,1], C^2(\mathbb{R}^N,\mathbb{R}^N)\right)}{\Phi(0)={\rm Id}, \quad |\Omega_t|=|\Omega| \quad 
    \forall
    t\in [0,1]},
\end{equation}
where $\Omega_t:=\Phi(t)(\Omega)=\setbld{\Phi(t)(x)}{x\in \Omega}$. 
For a given $\Phi\in \mathcal{A}$, we will use the following notation concerning its expansion:
\begin{equation}\label{expansion for Phi}
    \Phi(t)={\rm{Id}}+t h + o(t) \quad \text{as }t\to 0^+, 
\end{equation}
for some function $h\in C^2(\mathbb{R}^N,\mathbb{R}^N)$, where the expansion holds in the $C^2$ norm.
Notice that, since $\restr{\frac{d}{dt}}{t=0}|\Omega_t|=\int_
{\partial\Omega}h\cdot n$, the fact that $|\Omega_t|=|\Omega|$ for all $t$ leads to the following first-order volume preserving condition:
\begin{equation}\label{first order volume preserving condition}
   \int_{\partial\Omega}h\cdot n =0\quad \text{for all }\Phi\in \mathcal{A}. 
\end{equation}
\begin{remark}\label{extension remark}
Any function $h\in C^2(\mathbb{R}^N,\mathbb{R}^N)$ satisfying \eqref{first order volume preserving condition} can be ``completed" to an element $\Phi\in \mathcal{A}$ in the sense of \eqref{expansion for Phi} (see \cite[Remark 2.2]{cavallina2017locally} for an explicit construction).     
\end{remark}

As shown in \cite[Theorem 3.3]{cavallina2019stability}, for any $\Phi\in \mathcal{A}$, we have
\begin{equation}\label{T'}
    \restr{\frac{d}{dt}}{t=0} T_D(\Omega_t) = \int_{\partial\Omega} (\partial_n u)^2 \ (h\cdot n).
\end{equation}
As a result, $\restr{\frac{d}{dt}}{t=0} T_D(\Omega_t)=0$ holds true for all $\Phi\in \mathcal{A}$ if and only if the function $(\partial_n u)^2$ belongs to the orthogonal complement in $L^2(\partial\Omega)$ of the subspace of functions with zero average. In other words, $\Omega$ is a critical shape for $T_D$ under a volume constraint if and only if $(\partial_n u)^2$ is constant on $\partial\Omega$, which in turn holds if and only if $u$ is a solution to the overdetermined problem \eqref{odp} (as a consequence of Hopf's boundary lemma).  

Analogously, if $(D,\Omega)$ is a solution to \eqref{odp}, by arguing as in \cite[Theorem 4.2]{cavallina2019stability}, we get the following expression of the second order derivative:
\begin{equation}\label{T''}
\restr{\frac{d^2}{dt^2}}{t=0} T_D(\Omega_t) = 2c\int_{\partial\Omega} \partial_n u'\ (h\cdot n) + 2c\int_{\partial \Omega}\partial^2_{nn}u\ (h\cdot n)^2,  
\end{equation}
where $u'$ (the so-called \emph{shape derivative of }$u$) is the solution to the following boundary value problem:
\begin{equation}\label{u'}
    \begin{cases}
        -{\rm{div}}(\sigma \nabla u')=0\quad \text{in }\Omega,\\
        u'= -c\ (h\cdot n) \quad \text{on }\partial\Omega.
    \end{cases}
\end{equation}

Based on \eqref{T''}, we give the following definitions:
\begin{definition}\label{def critical shapes}
 Let $(D,\Omega)$ be a solution to \eqref{odp}. We say that 
 \begin{itemize}
     \item $(D,\Omega)$ is a \emph{local maximizer} if for all $\Phi\in\mathcal{A}$, $\restr{\frac{d^2}{dt^2}}{t=0} T_D(\Omega_t)\le 0$,
    \item $(D,\Omega)$ is a \emph{local minimizer} if for all $\Phi\in\mathcal{A}$, $\restr{\frac{d^2}{dt^2}}{t=0} T_D(\Omega_t)\ge0$,
    \item $(D,\Omega)$ is a \emph{saddle shape} if there exists two distinct perturbations in $\mathcal{A}$ for which the quantity $\restr{\frac{d^2}{dt^2}}{t=0} T_D(\Omega_t)$ has opposite signs.
 \end{itemize}
 Moreover, we will add the word ``strict" to the terminologies above, whenever we need to emphasize that the inequalities above are satisfied in the strict sense.
\end{definition}

It is known (see \cite[Theorem 4.5]{cavallina2019stability}) that concentric balls are local maximizers if $\sigma_c>1$ and saddle shapes if $\sigma_c<1$ in the sense of Definition \ref{def critical shapes}. Moreover, by a perturbation analysis relying on the implicit function theorem for Banach spaces \cite{cavallina&yachimura2020two, cavallina2022simultaneous} and a bifurcation analysis by the Crandall--Rabinowitz theorem \cite{cavallina&yachimura2021symmetry}, one can show the existence of many other nontrivial (non radially symmetric) local maximizers and saddle shapes $(D,\Omega)$.  

In what follows, we will show that local minimizers in the sense of Definition \ref{def critical shapes} actually do not exist, proving Theorem \ref{thm I}.

\begin{proof}[Proof of Theorem \ref{thm I}]
Consider the following Neumann-to-Dirichlet operator:
\begin{equation*}
\begin{aligned}
    \Lambda:\quad &L_*^2(\partial\Omega)\to H_*^{1/2}(\partial\Omega)\hookrightarrow L_*^2(\partial\Omega)\\
    & \qquad \xi \longmapsto \restr{v}{\partial\Omega},
\end{aligned}
    \end{equation*}
where the subscript $*$ indicates the subspace of functions with zero average on $\partial\Omega$ and $v\in H^1(\Omega)$ denotes the unique solution to the Neumann boundary value problem
\begin{equation}\label{neumann problem classical}
    \begin{cases}
        -{\rm{div}}(\sigma\nabla v)=0\quad \text{in }\Omega,\\
        \partial_n v=\xi\quad \text{on }\partial\Omega,
    \end{cases}
\end{equation} normalized so that $\int_{\partial\Omega} v =0$. In other words, $v$ is the unique function in $H_*^1(\Omega)$ that satisfies the following weak formulation:
\begin{equation}\label{neumann problem}
\int_\Omega \sigma \nabla v \cdot \nabla\varphi = \int_{\partial\Omega}\xi\varphi\quad \text{for all }\varphi\in H^1(\Omega).
\end{equation}
By the trace theorem \cite[Theorem 1.5.1.3]{grisvard2011elliptic}, $\Lambda$ is a well-defined bounded operator from $L_*^2(\partial\Omega)$ into $H_*^{1/2}(\partial\Omega)$. Thus, by the compactness of the embedding $H_*^{1/2}(\partial\Omega)\hookrightarrow L_*^2(\partial\Omega)$, $\Lambda$ becomes a compact operator from $L_*^2(\partial\Omega)$ into itself. 
Let us now show that $\Lambda:L_*^2(\partial\Omega)\to L_*^2(\partial\Omega)$ is self-adjoint. To this end, take two arbitrary functions $\xi,\eta\in L_*^2(\partial\Omega)$ and let $v,w\in H_*^1(\Omega)$ be the corresponding normalized solutions to the Neumann boundary value problem \eqref{neumann problem}. Taking $w$ and $v$ as test functions in the problems for $v$ and $w$ respectively yields
\begin{equation*}
    \langle\xi, \Lambda\eta\rangle_{L^2(\partial\Omega)} = \int_{\partial\Omega}\xi w = \int_\Omega \sigma \nabla v\cdot \nabla w = \int_{\partial\Omega} v \eta = \langle\Lambda\xi,\eta\rangle_{L^2(\partial\Omega)}, 
\end{equation*}
whence $\Lambda$ is a compact self-adjoint operator from $L_*^2(\partial\Omega)$ into itself. As a result, by the spectral theorem for compact self-adjoint operators \cite[Theorem 6.11]{brezis2011functional}, there exists an orthonormal basis of eigenfunctions $\{\xi_k\}_k\subset L_*^2(\partial\Omega)$ whose eigenvalues $(\lambda_k)_k$ converge to $0$ as $k\to \infty$ (see also \cite[Theorem 6.8]{brezis2011functional}).
If we set $v_k\in H_*^1(\Omega)$ to be the solution to the Neumann boundary value problem \eqref{neumann problem} for $\xi=\xi_k$, we have:
\begin{equation}\label{lambda xi = v}
  \lambda_k \xi_k = \Lambda \xi_k =\restr{v_k}{\partial\Omega}.  
\end{equation}
It is easy to show that $\lambda_k>0$ for all $k$. Indeed, be taking $v_k$ as a test function in \eqref{neumann problem},
\begin{equation*}\label{recap: who's who}
    \lambda_k = \langle \Lambda\xi_k,\xi_k\rangle_{L^2(\partial\Omega)}=\int_\Omega \sigma |\nabla v_k|^2 \ge 0. 
\end{equation*}
Now, if the right-hand side above were equal to zero, the function $v_k\in H_*^1(\Omega)$ would be constant. As a result, \eqref{neumann problem} would imply that $\xi_k\equiv 0$, in contrast with $\xi_k$ being an element of an orthonormal basis.

Finally, by following \cite[Theorem 2.5.1.1]{grisvard2011elliptic}, one can bootstrap the regularity of the functions $v_k$ to the class $W^{m,p}$ near the boundary $\partial\Omega$ for arbitrarily large $m$. As a result, $\xi_k\in C^\infty(\partial\Omega)$ and, by Remark \ref{extension remark}, the function $h_k:= \xi_k n \in C^\infty(\partial\Omega, \mathbb{R}^N)$ can be ``completed" to an element $\Phi_k\in\mathcal{A}$. 

Consider now the solution $u'_k$ to \eqref{u'} with respect to such $\Phi_k$. By comparing problems \eqref{u'} and \eqref{neumann problem classical} with \eqref{lambda xi = v} at hand, we get:
\begin{equation*}
    u'_k=-\frac{c}{\lambda_k}v_k\quad \text{in }\Omega.
\end{equation*}
Thus, by \eqref{lambda xi = v}, 
\begin{equation*}
 \partial_n u'_k = -\frac{c}{\lambda_k}\partial_n v_k = -\frac{c}{\lambda_k}\xi_k \quad \text{on }\partial\Omega.   
\end{equation*}
Finally, by \eqref{T''}, we can estimate the value of the second order derivative of $T_D(\Omega_t)$ corresponding to the perturbation $\Phi_k$ as follows:
\begin{equation*}
    \restr{\frac{d^2}{dt^2}}{t=0} T_D(\Omega_t) = 2c\int_{\partial\Omega} \frac{-c}{\lambda_k}\xi_k^2 + 2c\int_{\partial\Omega}\partial^2_{nn}u\ \xi_k^2\le -\frac{2c^2}{\lambda_k} + 2c \min_{\partial\Omega}\partial^2_{nn}u,  
\end{equation*}
where, in the last inequality, we have used the fact that $c<0$.
Notice that, for sufficiently large $k$, the right-hand side above becomes negative. Hence, $(D,\Omega)$ cannot be a local minimizer in the sense of Definition \ref{def critical shapes}, concluding the proof of Theorem \ref{thm I}.
\end{proof}

\section{Proof of Theorem \ref{thm II}}\label{sec pf II}

In this section, we will show Theorem \ref{thm II}. To this end, first, we will give the precise definition of ``\emph{having a tentacle}".

Let $(D,\Omega)$ be a solution to the overdetermined problem \eqref{odp} and, (if necessary) up to a change of coordinates, let $\Pi_\lambda$ denote the hyperplane
\begin{equation*}
    \Pi_\lambda:= \setbld{(\lambda, x_2,\dots, x_N)}{(x_2,\dots, x_N)\in \mathbb{R}^{N-1}}.
\end{equation*}
In what follows we are going to move the hyperplane $\Pi_\lambda$ by slowly decreasing the parameter $\lambda\in\mathbb{R}$. 
As a starting position, assume that $\lambda$ is large enough so that $\overline{\Omega}$ sits on the ``left" of $\Pi_\lambda$, that is $\overline\Omega\subset\setbld{(x_1,\dots,x_N)\in \mathbb{R}^N}{x_1<\lambda}$. Then it is possible to move $\Pi_\lambda$ to the ``left" by decreasing the value of $\lambda$ until $\partial\Omega$ and $\Pi_\lambda$ first touch for some $\lambda=\lambda_0$. Now, for even smaller values of $\lambda$, we obtain a nonempty portion of $\Omega$ lying on the ``right" side of $\Pi_\lambda$. We are interested in the reflection of this portion with respect to the hyperplane $\Pi_\lambda$:
\begin{equation*}
\Omega_\lambda:=\setbld{x_\lambda\in \mathbb{R}^N}{x=(x_1,\dots, x_N)\in \Omega, \quad x_1>\lambda},
\end{equation*}
where $x_\lambda$ denotes the reflection of the point $x$ with respect to the hyperplane $\Pi_\lambda$:
\begin{equation*}
    x_\lambda=(x_1,\dots, x_N)_\lambda:= (2\lambda-x_1, x_2, \dots, x_N).
\end{equation*}
Notice that, since $\partial\Omega$ is sufficiently smooth by Remark \ref{rmk boundary}, the set $\Omega_\lambda$ lies inside $\Omega\setminus\overline{D}$ if $\lambda$ is smaller but sufficiently close to $\lambda_0$. One can keep decreasing the value of $\lambda$ until (at least) one of the following happens:
\begin{enumerate}[label=(\roman*)]
    \item $\overline{\Omega_\lambda}$ touches $\partial \overline D$ at some point $P$.
    \item $\overline{\Omega_\lambda}$ touches $\partial \Omega$ at some point $P\notin\Pi_\lambda$.
    \item $\Pi_\lambda$ meets $\partial\Omega$ orthogonally at some point $P$.
\end{enumerate}

\begin{definition}\label{def tentacles}
We say that the pair $(D,\Omega)$ ``\emph{has a tentacle}" if there exists some change of coordinates such that either case $(ii)$ or case $(iii)$ above occurs. On the other hand, if, for all changes of coordinates, cases $(ii)-(iii)$ never occur, we say that the pair $(D,\Omega)$ ``\emph{has no tentacles}".
\end{definition}

\begin{figure}
    \centering
    \includegraphics[width=\linewidth]{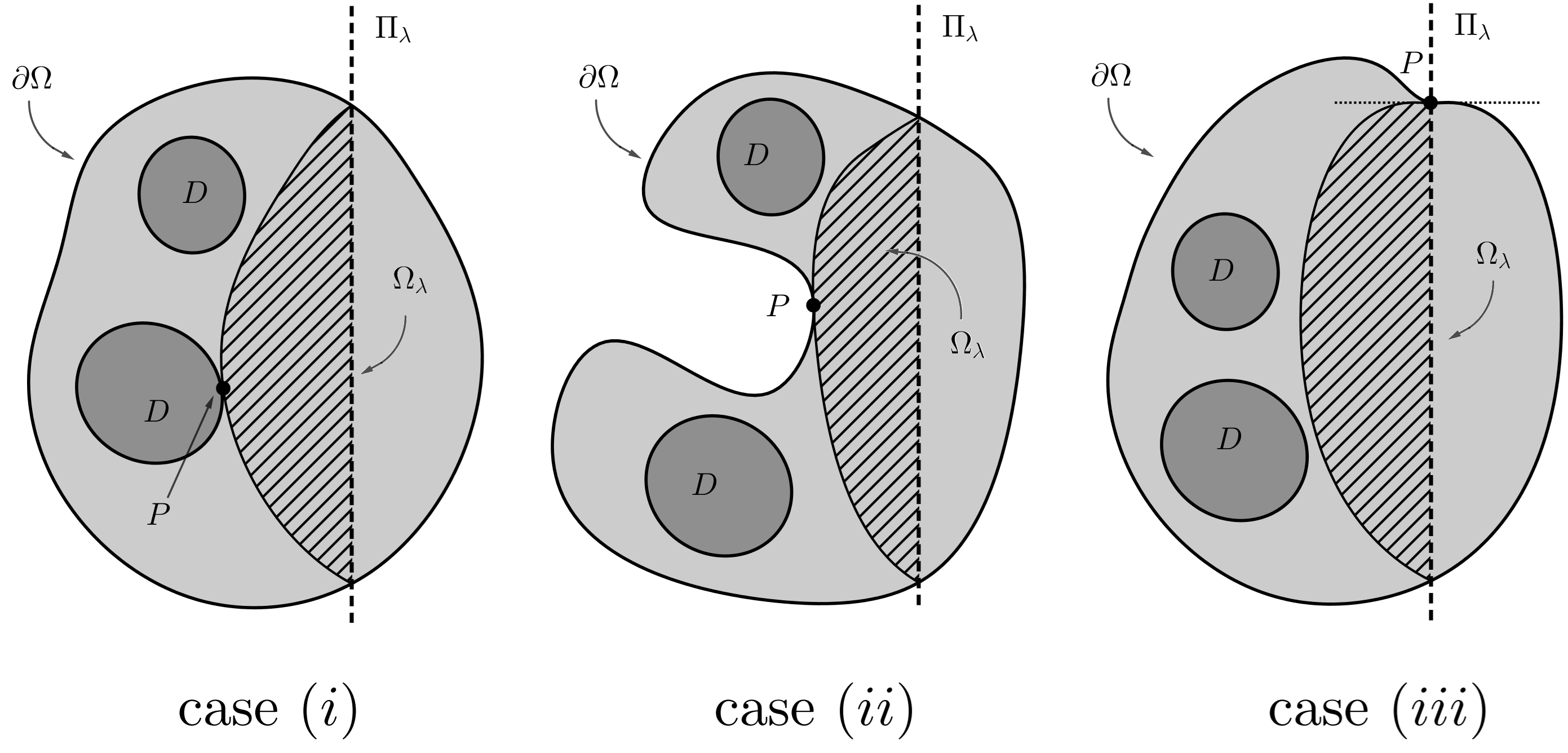}
    \caption{A depiction of the three cases described in Definition \ref{def tentacles}}
    \label{fig:threecases}
\end{figure}

\begin{proof}[Proof of Theorem \ref{thm II}]
By way of contradiction, suppose that there exists a change of coordinates for which either case $(ii)$ or case $(iii)$ happens. 
In either case, let $\lambda$ denote the exact value of the parameter at which either case $(ii)$ or case $(iii)$ happens and consider the following auxiliary function:
\begin{equation*}
w_\lambda:=u-u_\lambda\quad \text{in } \Omega_\lambda,
\end{equation*}
where 
\begin{equation*}
u_\lambda(x):=u(x_\lambda), \quad \text{for }x\in \Omega_\lambda.    
\end{equation*}

By construction, $w_\lambda$ solves the following boundary value problem:
\begin{equation}\label{eq w_la}
    \begin{cases}
        \Delta w_\lambda =0 \quad \text{in }\Omega_\lambda,\\
        w_\lambda=0\quad \text{on }\partial\Omega_\lambda\cap \Pi_\lambda,\\
        w_\lambda=u\ge 0\quad \text{on }\partial\Omega_\lambda\setminus\Pi_\lambda.
    \end{cases}
\end{equation}
Moreover, notice that, since $D\ne \emptyset$, $\Omega$ is not symmetric with respect to $\Pi_\lambda$, thus, there exists some point $Q\in \partial\Omega_\lambda\cap\Omega\setminus\Pi_\lambda$. In particular, $w_\lambda(Q)>0$ and thus $w_\lambda$ is strictly positive inside $\Omega_\lambda$. 

First, let us assume that case $(ii)$ happens.
In this case, the harmonic function $w_\lambda$ attains its minimum at the boundary point $P$, where it vanishes. Thus, by Hopf's boundary lemma, $\partial_n w_\lambda<0$ at $P$. On the other hand, since $u$ solves \eqref{odp} by assumption, we have
\begin{equation*}
    \partial_n w_\lambda(P)=\partial_n u(P) - \partial_n u(P_\lambda)= c-c=0,
\end{equation*}
leading to a contradiction.

 Let now assume that case $(iii)$ happens. 
In this case, Hopf's boundary lemma is insufficient and we will resort to Serrin's corner Lemma \cite[Lemma 1]{serrin1971}. We will show that, if case $(iii)$ happens, the point $P$ is a second-order zero for $w$, that is, $w$ and all its first and second-order derivatives vanish at $P$. If this is the case, then \cite[Lemma 1]{serrin1971} would imply $w\equiv 0$ in $\Omega_\lambda$, leading to a contradiction. 

We will now show that $P$ is a second order zero for the function $w_\lambda$. To this end, let us consider a coordinate system $(\xi,\eta)=(\xi_1,\dots, \xi_{N-1}, \eta)\in \mathbb{R}^{N-1}\times \mathbb{R}$ with the origin at $P$, the $\eta$ axis pointing in the direction $n$ and the $\xi_1$ axis orthogonal to $\Pi_\lambda$. In what follows, for any given scalar function $\varphi$, the ``full" gradient and Hessian matrix with respect to the $N$ variables $(\xi,\eta)=(\xi_1,\xi_2,\dots, \xi_{N-1}, \eta)$ will be denoted by $\nabla \varphi$ and $D^2\varphi$ respectively, while $\nabla_\xi \varphi$ and $D^2_\xi \varphi$ will denote the following:
\begin{equation*}
        \nabla_\xi \varphi:=\left(\partial_{\xi_1}\varphi,\partial_{\xi_2}\varphi,\dots, \partial_{\xi_{N-1}}\varphi \right), \quad D_\xi^2 \varphi:= \left(\partial^2_{\xi_i, \xi_j}\varphi \right)_{1\le i,j\le N-1}.
\end{equation*}

There exists a smooth function $f:\mathbb{R}^{N-1}\to \mathbb{R}$ such that a small portion of $\partial\Omega$ near $P$ can be locally expressed as the graph of $f$, as follows: 
\begin{equation*}
\setbld{\big(\xi, f(\xi)\big)\in \mathbb{R}^{N-1}\times\mathbb{R}}{ \vert\xi\rvert <\varepsilon}.
\end{equation*}
In particular, since we are assuming that $\Pi_\lambda$ meets $\partial\Omega$ orthogonally at $P$, we have 
\begin{equation}\label{gr f =0}
  \nabla_\xi f(0)=0.
\end{equation}
Using the local parametrization introduced above, the condition $u\equiv 0$ on $\partial\Omega$ can be rewritten as 
\begin{equation}\label{u=0}
u(\xi,f(\xi))=0,\quad \text{for }|\xi|<\varepsilon.
\end{equation}
Now, differentiating \eqref{u=0} with respect to $\xi$ yields 
\begin{equation}\label{u=0 diff}
    \nabla_\xi u(\xi,f(\xi))+\partial_{\eta} u(\xi, f(\xi))\nabla_\xi f(\xi)=0.
\end{equation}
Computing \eqref{u=0 diff}
at the point $P$ with \eqref{gr f =0} at hand yields $\nabla_\xi u(P)=0$. This fact, combined with $c=\partial_\eta u(P)=\partial_\eta u_\lambda(P)$, ensures that all first derivatives of $u$ and $u_\lambda$ coincide at $P$, thus $\nabla w_\lambda(P)=0$, as claimed.  

In what follows, we will show that also $D^2 w(P)=0$. To this end, notice that in the new coordinate system, we have
\begin{equation*}
u_\lambda(\xi_1, \xi_2\dots, \xi_{N-1},\eta)=u(-\xi_1,\xi_2,\dots, \xi_{N-1},\eta).
\end{equation*}
In particular, this implies
\begin{equation*}
\begin{aligned}
\partial^2_{ \ell\ell} u(P)=\partial^2_{\ell \ell} u_\lambda(P) \quad \text{for }\ell\in\{\xi_1,\dots, \xi_{N-1}, \eta\},\\
\partial_{\xi_1}\partial_{ \ell} u(P)=-\partial_{\xi_1}\partial_{ \ell} u_\lambda(P) \quad \text{for }\ell\in\{\xi_2,\dots, \xi_{N-1}, \eta\}.
\end{aligned}
\end{equation*}
We are now left to show that all mixed derivatives with respect to $\xi_1$ and $\ell\in\{\xi_2,\dots, \xi_{N-1}, \eta\}$ of $u$ at $P$ vanish.
To this end, differentiate \eqref{u=0 diff} with respect to $\xi$ to get 
\begin{equation*}
   D^2_\xi u+ 2 \partial_\eta \nabla_\xi u\otimes \nabla_\xi f + \partial^2_{\eta \eta } u \nabla_\xi f \otimes \nabla_\xi=0. 
\end{equation*}
Evaluating this at the point $P$ with \eqref{gr f =0} at hand yields the desired 
\begin{equation*}
    D^2_\xi u(P)=0. 
\end{equation*}
This leaves us with just one equality to show, that is $\partial_{\xi_1}\partial_\eta u(P)=0$. To this end, notice that the outward normal $n$ admits the local expression
\begin{equation*}
n(\xi,f(\xi))=\frac{\left(-\nabla_\xi f(\xi),1\right)}{\sqrt{1+|\nabla_\xi f(\xi)|^2}}, \quad\text{for } |\xi|<\varepsilon.
\end{equation*}
Therefore, the overdetermined condition $\partial_n u\equiv c$ on $\partial\Omega$ in local coordinates reads
\begin{equation}\label{dnuc}
-\nabla_\xi u(\xi,f(\xi))\cdot\nabla_\xi f(\xi)+\partial_\eta u(\xi,f(\xi))=c\sqrt{1+|\nabla_\xi f(\xi)|^2}, \quad\text{for } |\xi|<\varepsilon.
\end{equation}
Differentiating \eqref{dnuc} with respect to $\xi$ yields
\begin{equation*}
-D^2_\xi u \nabla_\xi f - (\partial_\eta \nabla_\xi u \otimes \nabla_\xi f) \nabla_\xi f + \nabla_\xi \partial_\eta u + \partial^2_{\eta \eta}u \nabla_\xi f = \frac{c D^2_\xi f \nabla_\xi f}{\sqrt{1+|\nabla_\xi f|^2}}. 
\end{equation*}
Finally, evaluating the above at $P$ with \eqref{gr f =0} at hand yields $\nabla_\xi \partial_\eta u(P)=0$, whence, in particular, the desired $\partial_{\xi_1}\partial_\eta u(P)=0$.
To summarize, in the above, we have shown that all first and second-order derivatives of $u$ and $u_\lambda$ coincide at $P$, whence 
\begin{equation*}
    \nabla w_\lambda(P), \quad w_\lambda(P)=0, \quad D^2 w_\lambda(P)=0.
\end{equation*}
As remarked before, by Serrin's corner Lemma \cite[Lemma 1]{serrin1971}, this is only possible if $w_\lambda\equiv 0$ in $\Omega_\lambda$. As this contradicts the fact that $w_\lambda$ is strictly positive inside $\Omega_\lambda$, the proof is completed.
\end{proof}

\section{Proof of Theorem \ref{thm III}}\label{sec pf III}
In what follows, we will show Theorem \ref{thm III}. Let $(D,\Omega)$ be a solution to the overdetermined problem \eqref{odp}. Recall that, by Remark \ref{rmk boundary}, $\partial\Omega$ is an analytic surface. 

If $\partial\Omega$ were to contain a flat part, then, by unique continuation, $\partial\Omega$ would contain a hyperplane, contradicting the fact that $\Omega$ is bounded. 

Similarly, if $\partial\Omega$ were to contain a portion of a sphere, then it would contain an entire sphere $S$ as a subset (where both $u\equiv 0$ and $\partial_n u \equiv c$ hold). Now, by the Cauchy--Kovalevskaya theorem \cite[Subsection 4.6.3]{Evans2010PartialDifferentialEquations}, $u$ is radial (with respect to the center of $S$, which, without loss of generality, we will assume to be the origin) in a neighborhood of $S$. Then, since $u$ is real analytic in the connected open set $\Omega\setminus\overline{D}$, $u$ must be radial in the whole $\Omega\setminus\overline{D}$. As $\partial \Omega$ is a level set of $u$, this leaves us with two possibilities:
\begin{equation*}
(i)\quad \text{$\Omega$ is a ball $\{|x|<R\}$} \quad \text{or}\quad (ii)\quad\text{$\Omega$ is an annulus $\{R_1<|x|<R_2\}$.}    
\end{equation*}

First, we will show that $(ii)$ cannot occur. Indeed, if $(ii)$ were true, the function $U(|x|):=u(x)$ would solve
\begin{equation*}
U''+\frac{N-1}{r}U' = -1 \quad \text{in }(R_1,R_2), \quad U(R_1)=U(R_2)=0.
\end{equation*}
Thus, by elementary computations, for any $x\in \overline{\Omega}\setminus D$ we would have:
\begin{equation*}
u(x)= U(|x|)= \begin{cases}
    \frac{R_1^2-R_2^2}{2N(R_1^{2-N}-R_2^{2-N})}|x|^{2-N}+\frac{R_1^2 R_2^2 (R_1^{-N}-R_2^{-N})}{2N(R_1^{2-N}-R_2^{2-N})}- \frac{|x|^2}{2N}\quad &\text{if }N\ge 3,\\[7pt]
    \frac{R_1^2 R_2^2}{4\log \frac{R_1}{R_2}}\log |x| + \frac{R_2^2\log R_1 - R_1^2 \log R_2}{4 \log \frac{R_1}{R_2}}-\frac{|x|^2}{4}\quad &\text{if }N=2.
\end{cases}
\end{equation*}
Nevertheless, it is easy to check that the function above does not satisfy 
\begin{equation*}
\restr{\partial_n u}{|x|=R_1} =\restr{\partial_n u}{|x|=R_2}    
\end{equation*}
for any choice of $0<R_1<R_2$, contradicting the fact that $u$ must solve \eqref{odp}.

As a result, the only possibility left is that $\Omega$ is a ball, in which case we conclude by Sakaguchi's symmetry result \cite[Theorem 5.1]{sakaguchibessatsu}. 

\section{Proof of Theorem \ref{thm IV}}\label{sec pf IV}
In what follows, we will show Theorem \ref{thm IV}. 
Assume that there exist two distinct positive constants $\alpha\ne \beta$ such that the overdetermined problem \eqref{odp} is solvable in the same pair $(D,\Omega)$ for both values $\sigma_c=\alpha,\beta$. Let $u_\alpha, u_\beta\in H_0^1(\Omega)$ denote the respective solutions to the following. 
\begin{equation}\label{u_al u_be weak form}
\begin{aligned}
  \alpha\int_D \nabla u_\alpha\cdot\nabla \varphi + \int_{\Omega\setminus\overline{D}} \nabla u_\alpha\cdot \nabla\varphi = \int_\Omega \varphi \quad \text{for all }\varphi\in H_0^1(\Omega), \\
  \beta\int_D \nabla u_\beta\cdot\nabla \varphi + \int_{\Omega\setminus\overline{D}} \nabla u_\beta\cdot \nabla\varphi= \int_\Omega \varphi \quad \text{for all }\varphi\in H_0^1(\Omega). 
\end{aligned}
\end{equation}
As \eqref{odp} is solvable by assumption, both $u_\alpha$ and $u_\beta$ satisfy the overdetermined condition 
\begin{equation}\label{dnu_a=dnu_b=c}
    \partial_n u_\alpha\equiv\partial_n u_\beta\equiv c \quad \text{on } \partial\Omega.
\end{equation}
As a result of \eqref{dnu_a=dnu_b=c}, since $\partial\Omega$ is an analytic surface by Remark \ref{rmk boundary}, the Cauchy--Kovalevskaya theorem \cite[Subsection 4.6.3]{Evans2010PartialDifferentialEquations} and the identity theorem for real analytic functions imply that $u_\alpha$ and $u_\beta$ coincide inside the connected set $\Omega\setminus\overline{D}$. In particular, since $u_\alpha$ and $u_\beta$ are continuous inside $\Omega$ by \cite[Theorem 8.22]{GilbargTrudinger}, equality extends to their traces on $\partial(\Omega\setminus\overline{D})\cap \Omega=\partial\overline{D}=\partial D$: 
\begin{equation}\label{traces are equal}
    \restr{u_\alpha}{\partial D}=\restr{u_\beta}{\partial D}.
\end{equation} 
Now, consider the function $v:=\alpha u_\alpha-\beta u_\beta\in H_0^1(\Omega)$. By subtracting both rows in \eqref{u_al u_be weak form} and taking $\varphi=v$, we get
\begin{equation*}
    \int_D |\nabla v|^2=0,
\end{equation*}
whence $v$ is locally constant in $D$. 
Following the definition of $v$, we get  
\begin{equation}\label{trace of v}
\restr{v}{\partial D} = \restr{(\alpha u_\alpha-\beta u_\beta)}{\partial D} = (\alpha-\beta)\restr{u_\alpha}{\partial D},
\end{equation}
where we have used \eqref{traces are equal} in the last equality.

Consider now the function $w\in H_0^1(\Omega)$ defined as
\begin{equation*}
    w:= u_\alpha + (\alpha-1)E_D\left(\restr{u_\alpha}{D}-\frac{1}{\alpha-\beta}\restr{v}{D}\right),
\end{equation*}
where $E_D: H_0^1(D)\hookrightarrow H_0^1(\Omega)$ denotes the extension-by-zero operator (see \cite[3.26 and the subsequent results]{adams2003sobolev}).
As $v$ is locally constant in $D$, we have
\begin{equation*}
    \nabla w = \begin{cases}
        \alpha\nabla u_\alpha\quad \text{in } D,\\
        \nabla u_\alpha\quad \text{in } \Omega\setminus\overline{D}.
    \end{cases}
\end{equation*}
Therefore, by the first identity in \eqref{u_al u_be weak form} it is immediate to see that \begin{equation*}
    \int_\Omega \nabla w \cdot\nabla\varphi = \int_\Omega \varphi \quad \text{for all }\varphi\in H_0^1(\Omega).
\end{equation*}
 Also, since $w$ coincides with $u_\alpha$ in $\Omega\setminus\overline{D}$, we have $\partial_n w\equiv \partial_n u_\alpha \equiv c$ on $\partial\Omega$. In other words, $\Omega$ solves the one-phase Serrin's overdetermined problem. As a result, $\Omega$ is a ball (of radius $R=-cN$) and thus
\begin{equation}\label{u_a u_b outside}    u_\alpha(x)=u_\beta(x)=w(x)=\frac{R^2-|x|^2}{2N}\quad \text{for }x\in \Omega\setminus\overline{D}.
\end{equation}
Moreover, by continuity, the expression in \eqref{u_a u_b outside} holds true up to the boundary $\partial(\Omega\setminus\overline{D})=\partial\Omega\cup \partial\overline{D}=\partial\Omega\cup\partial D$.  
Combining this with \eqref{trace of v} yields:
\begin{equation}\label{trace v radial}
    v(x)= (\alpha-\beta)\frac{R^2-|x|^2}{2N}\quad \text{for }x\in\partial D.
\end{equation}
Recall now that the continuous function $v$ is locally constant in $D$, thus being constant on the boundary of each connected component of $D$. By \eqref{trace v radial}, the boundary of any connected component of $D$ lies inside some sphere concentric with $\partial\Omega$. We conclude that $D$ must have exactly one connected component: a ball concentric with $\Omega$. This concludes the proof of the theorem.

\bibliographystyle{ieeetr}
\bibliography{references}

\begin{thebibliography}{10}

\bibitem{serrin1971}
J.~Serrin, ``A symmetry problem in potential theory,'' {\em Arch. Rational
  Mech. Anal.}, vol.~43, pp.~304--318, 1971.

\bibitem{kenig1994harmonic}
C.~E. Kenig, {\em Harmonic analysis techniques for second order elliptic
  boundary value problems}, vol.~83.
\newblock American Mathematical Soc., 1994.

\bibitem{kinderlehrer1977regularity}
D.~Kinderlehrer and L.~Nirenberg, ``Regularity in free boundary problems,''
  {\em Annali della Scuola Normale Superiore di Pisa-Classe di Scienze},
  vol.~4, no.~2, pp.~373--391, 1977.

\bibitem{sokolonikoffElasticity}
I.~S. Sokolonikoff, {\em Mathematical Theory of Elasticity}.
\newblock McGraw-Hill, 1956, New York.

\bibitem{dondl2024phase}
P.~Dondl, A.~Maione, and S.~Wolff-Vorbeck, ``Phase field model for
  multi-material shape optimization of inextensible rods,'' {\em ESAIM:
  Control, Optimisation and Calculus of Variations}, vol.~30, p.~50, 2024.

\bibitem{sakaguchi2016rendiconti}
S.~Sakaguchi, ``Two-phase heat conductors with a stationary isothermic
  surface,'' {\em Rend. Istit. Mat. Univ. Trieste}, vol.~48, pp.~167--187,
  2016.

\bibitem{sakaguchibessatsu}
S.~Sakaguchi, ``Two-phase heat conductors with a stationary isothermic surface
  and their related elliptic overdetermined problems,'' {\em RIMS
  K{\^o}ky{\^u}roku Bessatsu}, vol.~B80, pp.~113--132, 2020.

\bibitem{CavallinaMagnaniniSakaguchi2021constantflowproperty}
L.~Cavallina, R.~Magnanini, and S.~Sakaguchi, ``Two-phase heat conductors with
  a surface of the constant flow property,'' {\em The Journal of Geometric
  Analysis}, vol.~31, no.~1, pp.~312--345, 2021.

\bibitem{cavallina2019stability}
L.~Cavallina, ``Stability analysis of the two-phase torsional rigidity near a
  radial configuration,'' {\em Applicable Analysis}, vol.~98, no.~10,
  pp.~1889--1900, 2019.

\bibitem{cavallina2017locally}
L.~Cavallina, ``Locally optimal configurations for the two-phase torsion
  problem in the ball,'' {\em Nonlinear Analysis}, vol.~162, pp.~33--48, 2017.

\bibitem{cavallina&yachimura2020two}
L.~Cavallina and T.~Yachimura, ``On a two-phase serrin-type problem and its
  numerical computation,'' {\em ESAIM: Control, Optimisation and Calculus of
  Variations}, vol.~26, p.~65, 2020.

\bibitem{cavallina2022simultaneous}
L.~Cavallina, ``The simultaneous asymmetric perturbation method for
  overdetermined free boundary problems,'' {\em Nonlinear Analysis}, vol.~215,
  p.~112685, 2022.

\bibitem{cavallina&yachimura2021symmetry}
L.~Cavallina and T.~Yachimura, ``Symmetry breaking solutions for a two-phase
  overdetermined problem of serrin-type,'' in {\em Current Trends in Analysis,
  its Applications and Computation: Proceedings of the 12th ISAAC Congress,
  Aveiro, Portugal, 2019}, pp.~433--441, Springer, 2021.

\bibitem{grisvard2011elliptic}
P.~Grisvard, {\em Elliptic Problems in Nonsmooth Domains}.
\newblock Classics in Applied Mathematics, Society for Industrial and Applied
  Mathematics, 2011.

\bibitem{brezis2011functional}
H.~Br{\'e}zis, {\em Functional analysis, Sobolev spaces and partial
  differential equations}, vol.~2.
\newblock Springer, 2011.

\bibitem{Evans2010PartialDifferentialEquations}
L.~C. Evans, {\em Partial Differential Equations}, vol.~19 of {\em Graduate
  Studies in Mathematics}.
\newblock American Mathematical Society, 2nd~ed., 2010.

\bibitem{GilbargTrudinger}
D.~Gilbarg and N.~S. Trudinger, {\em Elliptic Partial Differential Equations of
  Second Order}.
\newblock Classics in Mathematics, Springer Berlin, Heidelberg, 2nd~ed., 2001.

\bibitem{adams2003sobolev}
R.~A. Adams and J.~J.~F. Fournier, {\em Sobolev spaces}.
\newblock Elsevier, 2003.

\end{thebibliography}

\noindent
\textsc{
Mathematical Institute, Tohoku University, Sendai 980-8578, Japan}\\
\noindent
{\em Electronic mail address:}
cavallina.lorenzo.e6@tohoku.ac.jp

\end{document}